\documentclass[11pt]{amsart}
\usepackage{amsmath,amssymb,amsthm,bm,xcolor,enumerate,caption}
\usepackage[mathscr]{euscript}

\newtheorem{theorem}{Theorem}[section]

\newtheorem{lemma}[theorem]{Lemma}
\newtheorem{proposition}[theorem]{Proposition}
\newtheorem{corollary}[theorem]{Corollary}

\theoremstyle{definition}

\newtheorem{remark}[theorem]{Remark}

\newtheorem{problem}{Problem}
\newtheorem{construction}[theorem]{Construction}
\newtheorem{definition}[theorem]{Definition}
\def\D{{\mathscr{D}}}
\def\PP{{\mathscr{P}}}

\def\B{{\mathscr{B}}}
\def\C{{\mathscr{C}}}
\def\Aut{{\mathrm{Aut}}}
\def\Sym{{\mathrm{Sym}}}

\def\Ch{\bm{\mathscr{C}}}
\def\Dft{\mathscr{D}^{\rm ft}}

\usepackage{anysize}
\marginsize{2cm}{2cm}{2cm}{2cm}

\title{Chain-imprimitive, flag-transitive $2$-designs}
\author{Carmen Amarra$^{\rm\lowercase{a,b\ast}}$} \author{Alice Devillers$^{\rm \lowercase{a}}$} \author{Cheryl E. Praeger$^{\rm \lowercase{a}}$}
\address{$^{\rm\lowercase{a}}$Centre for the Mathematics of Symmetry and Computation, The University of Western Australia, 35 Stirling Highway, Crawley, Western Australia 6009, Australia}
\email{alice.devillers@uwa.edu.au, cheryl.praeger@uwa.edu.au}
\address{$^{\rm\lowercase{b}}${\rm{\emph{P\lowercase{ermanent address}}}}: Institute of Mathematics, University of the Philippines Diliman, C.P. Garcia Avenue, Quezon City 1101, Philippines}
\email{mcamarra@math.upd.edu.ph}
\address{$^{\ast}$Corresponding author}

\begin{document}

\maketitle
\begin{abstract}
We consider $2$-designs which admit a group of automorphisms that is flag-transitive and leaves invariant a chain of nontrivial point-partitions. We build on our recent work on $2$-designs which are block-transitive but not necessarily flag-transitive. In particular we use the concept of the ``array'' of a point subset with respect to the chain of point-partitions; the array describes the distribution of the points in the subset among the classes of each partition. We obtain necessary and sufficient conditions on the array in order for the subset to be a block of such a design. By explicit construction we show that for any $s \geq 2$, there are infinitely many $2$-designs admitting a flag-transitive group that preserves an invariant chain of point-partitions of length $s$. Moreover an exhaustive computer search, using {\sc Magma}, seeking designs with  $e_1e_2e_3$ points (where  each $e_i\leq 50$) and a partition chain of length $s=3$,  produced $57$ such flag-transitive designs, among which only three designs arise from our construction -- so there is still much to learn. 
\end{abstract}

\section{Introduction}

We investigate $2$-designs which admit an automorphism group that preserves a chain structure on its point set. Thus we are interested in $2$-$(v,k,\lambda)$ designs $\D$, namely incidence structures that consist of a \emph{point set} $\PP$ and a \emph{block set} $\B$ of subsets of $\PP$ with the property that $|\PP| = v$,  each block has size $k$, and any two distinct points are contained in exactly $\lambda$ blocks. A subgroup $G \leq \Aut(\D)$ is said to act \emph{$s$-chain-imprimitively} on $\PP$ if $G$ is transitive on $\PP$ and leaves invariant each partition in an $s$-chain
    \begin{equation} \label{eq:chain}
    \Ch : \ {\textstyle\binom{\PP}{1}} = \C_0 \prec \C_1 \prec \ldots \prec \C_{s-1} \prec \C_s = \{\PP\}
    \end{equation}
of partitions of $\PP$. (Here $\binom{\PP}{1}$ denotes the partition of $\PP$ into singletons, and $\C_{i-1} \prec \C_i$ means that each $\C_{i-1}$-class is a proper subset of some $\C_i$-class.) Any point-imprimitive design is $2$-chain-imprimitive and recent results in \cite{chainspaper}, that both characterised parameters and gave explicit examples, showed that \emph{there exist $G$-block-transitive $2$-$(v,k,\lambda)$ designs $\D=(\PP,\B)$ with $G$ acting $s$-chain-imprimitively on $\PP$ for which the chain length $s$ is arbitrarily large}. A question arose from the work in \cite{chainspaper} as to whether the same assertion would hold true if we required also that $G$ should be flag-transitive, not only block-transitive. Confirming this assertion is the aim of this paper.

A  \emph{flag} of a $2$-$(v,k,\lambda)$ design $\D=(\PP,\B)$ is a pair $(\delta,B) \in \PP \times \B$ such that $\delta \in B$, and $\D$ is \emph{$G$-flag-transitive} (for $G\leq \Aut(\D)$) if $G$ acts transitively on the set of flags. Flag-transitivity of $G$ implies that $G$ is both block-transitive and point-transitive, but the converse does not hold in general. Moreover, Davis \cite{D87} showed that for each value of the parameter $\lambda$ there are only finitely many possible $\D$ with $G$ point-imprimitive and flag-transitive. So we were not at all sure when we began this study that we would find examples of $G$-flag-transitive designs with $G$ acting $s$-chain-imprimitively on $\PP$ and the chain length $s$ unboundedly large. None of the block-transitive designs constructed in \cite{chainspaper} were flag-transitive, and although there are many papers in the literature contributing to the classification of flag-transitive, point-imprimitive $2$-designs, for example \cite{grids21,CP16,CZ,DLPX,DP21,DP22,LPR,MS,M,M2,OR,P,PZ,ZZ}, we did not find any constructions in them which would help to answer this question. Building on the theory developed in \cite{chainspaper} we were able to prove the following result. 

\begin{theorem} \label{thm:ex-ft}
For any integer $s \geq 2$, there exist infinitely many $2$-designs $\D=(\PP,\B)$ such that some subgroup  $G\leq \Aut(\D)$ acts flag-transitively on  $\D$ and $s$-chain-imprimitively on $\PP$.
\end{theorem}

In order to prove this theorem, which ultimately relies on an explicit construction (see Constructions \ref{constr:gen} and \ref{constr:family}), we needed to develop further the theoretical framework introduced in \cite{chainspaper}. Suppose that $\D=(\PP,\B)$ is a $2$-$(v,k,\lambda)$ design and $G\leq \Aut(\D)$ is flag-transitive on  $\D$ and $s$-chain-imprimitive on $\PP$, preserving the partitions in an $s$-chain $\Ch$ as in \eqref{eq:chain}. Parameters of significance are the positive integers $e_1,\dots,e_s$ such that, for $1 \leq i \leq s$, each class of $\C_i$ contains $e_i$ classes of $\C_{i-1}$. Then $v=e_1\dots e_s$, and in the next result Theorem~\ref{thm:ft} we identify necessary and sufficient conditions on the parameters $v,k, s, e_1,\dots,e_s$ for the existence of a flag-transitive, $s$-chain-imprimitive $2$-design. 

\begin{theorem} \label{thm:ft}
Let $s, k, e_1, \ldots, e_s$ be integers, all at least $2$, such that $k < v$, where $v := \prod_{i=1}^s e_i$. Let  $\PP$ be a set of  size $v$ and  $\Ch$ a chain of partitions as in \eqref{eq:chain}, such that for $1 \leq i \leq s$ each class of $\C_i$ contains $e_i$ classes of $\C_{i-1}$. Let $d := \gcd(e_1 - 1, \ldots, e_s - 1)$, and for each $i \in \{1, \ldots, s\}$ let       
    \begin{equation} \label{def:yi}
    y_i := 1 + \frac{k-1}{v-1} \left( \left(\prod_{j=1}^i e_j\right) - 1 \right) \quad \mbox{(so in particular $y_s = k$)}.
    \end{equation}
Then there exists a $2$-$(v,k,\lambda)$ design $\D$ with point set $\PP$, for some $\lambda$, and a group $G\leq \Aut(\D)$ such that $G$ is flag-transitive on $\D$ and preserves the partitions in the $s$-chain $\Ch$, if and only if the following conditions both hold:
    \begin{itemize}
    \item[(FT1)] \label{flagtr1} $v-1$ divides $(k-1) \cdot d$; and 
    \item[(FT2)] \label{flagtr2} for each $i \in \{1, \ldots, s-1\}$, $y_{i}$ is a positive integer dividing $(e_{i+1} - 1)\left(\prod_{j=1}^{i} e_j\right)/d$.
    \end{itemize}
Moreover if these two conditions hold and $\D = (\PP,\B)$ is such a design then, for each $i \in \{1, \ldots, s\}$, each class $C \in \C_i$ and  each block $B \in \B$, the intersection size $|B \cap C| \in \{0,y_i\}$.
\end{theorem}

In Theorem~\ref{thm:ft}, the flag-transitive group $G$ is contained in the full stabiliser  of the partition chain $\Ch$, namely the iterated wreath product $W=S_{e_1} \wr \ldots \wr S_{e_s}$. Thus (see, for example, \cite[Proposition 1.1]{CP93}), if we replace  $G$ by its over-group $W$, and replace the block set $\B$ (which is the $G$-orbit $B^G=\{B^g\mid g\in G\}$ of a subset $B \subseteq \PP$) by the possibly larger $W$-orbit $B^W$, then $\D'=(\PP, B^W)$ is also a $2$-$(v,k,\lambda')$ design with $W\leq \Aut(\D')$ flag-transitive and preserving the same $s$-chain $\Ch$ of partitions of $\PP$. So $\D'$ will have the same symmetry properties,  the same point set $\PP$ and partition chain $\Ch$,  the same parameters $v,k, s, e_1,\dots,e_s$, and possibly larger $\lambda'\geq \lambda$. Thus, in order to identify in Theorem~\ref{thm:ft} necessary and sufficient conditions on the parameters $v,k, s, e_1,\dots,e_s$ for the existence of a flag-transitive, $s$-chain-imprimitive $2$-design, we may work with the larger group $W$ and design $\D'$. Similarly, in order to prove Theorem \ref{thm:ex-ft}, we find, for each $s$, infinitely many parameter sequences $k, e_1,\dots,e_s$ satisfying the conditions (i) and (ii) of Theorem~\ref{thm:ft}. Moreover, we explicitly exhibit a $k$-subset $B$ satisfying the conditions in the final sentence of Theorem~\ref{thm:ft}, and prove that $(\PP, B^W)$ is a $2$-design with the required properties (see Constructions \ref{constr:gen} and \ref{constr:family}; and Propositions~\ref{prop:gen-ft} and \ref{prop:ft}).

We end this section with some comments in Subsection~\ref{sub:comments} about the design constructions and some open questions. The organisation of the rest of paper is as follows: Section~\ref{s:prelims} contains necessary background work about iterated wreath products $W$ and their action on $\PP$. Section~\ref{sec:arrays} contains some specific results about `uniform' point-sets, which are used in Section~\ref{sec:gencon} to justify our general construction method for flag-transitive $s$-chain-imprimitive $2$-designs given in Construction~\ref{constr:gen}. In Section~\ref{sec:flagtr} we prove Theorem~\ref{thm:ft}, and in the final Section~\ref{sec:expl} we present explicit families of examples, and in particular we prove Theorem~\ref{thm:ex-ft}.

\subsection{Comments on our constructions and results} \label{sub:comments}

We make several comments about constructions of flag-transitive $s$-chain-imprimitive $2$-designs, and in particular relate our work to earlier results.

\smallskip

(a) \emph{Comment on Theorem~\ref{thm:ft} for the case $s=2$.}\quad Theorem~\ref{thm:ft} for the case $s=2$ can be derived from several results in \cite{CP93} as follows. If $s=2$ then the number of points is $v=|\PP|=e_1e_2$.  The chain $\Ch$ in \eqref{eq:chain} has just one nontrivial partition, namely $\C_1=\{C_1,\dots,C_{e_2}\}$, with $e_2$ classes of size $e_1$. The point-block incidence structure $\D=\D(e_1,e_2; \mathbf{x}) = (\PP, \B)$ defined in \cite[Section 2]{CP93}, where $\mathbf{x}=(x_1,\dots, x_{e_2})$ with non-negative integer entries $x_1 \geq \ldots \geq x_{e_2}$, has as block set $\B$ the collection of all subsets $B\subseteq \PP$ such that the sequence $\big(\, |B\cap C_i| \ | \ 1\leq i\leq e_2 \,\big)$ is a permutation of the sequence $\mathbf{x}$. Thus $k:=|B|=\sum_i x_i$, and $\Aut(\D)$ is the wreath product $S_{e_1}\wr S_{e_2}$. As noted at the beginning of the proof of \cite[Proposition 4.1]{CP93}, a necessary and sufficient condition for $\D(e_1,e_2; \mathbf{x})$ to be flag-transitive is that the tuple $\mathbf{x}$ is some rearrangement of $(\ell^{k/\ell},0^{e_2-k/\ell})$, for a divisor $\ell$ of $k$ with $1<\ell<k$. Then the criterion given by \cite[Proposition 2.2(ii)]{CP93} for $\D$ to be a $2$-design is that $k(\ell-1) = k(k-1)(e_1-1)/(v-1)$, that is $\ell=1+(k-1)(e_1-1)/(v-1)$. Thus $\ell$ is the parameter $y_1$ in \eqref{def:yi}. Since $\ell$ is a positive integer this equation implies that $v-1$ divides $(k-1)(e_1-1)$. Note that $\gcd(v-1, e_1-1)=\gcd(e_1e_2-1, e_1-1)=\gcd(e_2-1, e_1-1)$, which is the parameter $d$ of Theorem~\ref{thm:ft}. Hence $v-1$ divides $(k-1)d$ which is condition (i) of Theorem~\ref{thm:ft}. To obtain condition (ii), we first subtract the equation $\ell=1+(k-1)(e_1-1)/(v-1)$ from the equation $k=1+(k-1)(e_1e_2-1)/(v-1)$ to obtain $k-\ell = z$, where $z=(k-1)(e_2-1)e_1/(v-1)$. Note that $z$ is a positive integer, and as $\ell$ divides $k$, $\ell$ must also divide $z$. Also $\gcd(e_2-1, v-1)=\gcd(e_2-1, e_1-1)=d$ and $v-1$ is coprime to $e_1$, so $z$ factorises as 
    \[
    z= \frac{(e_2-1)e_1}{d}\cdot \frac{k-1}{(v-1)/d}.
    \]
Then, again as $\ell$ divides $k$, $\ell$ is coprime to the second factor and hence $\ell$ divides $(e_2-1)e_1/d$, which is condition (ii). Thus existence of a flag-transitive point-imprimitive $2$-design implies conditions (i) and (ii), and reversing this argument we see that if conditions (i) and (ii) hold then the design  $\D(e_1,e_2; \mathbf{x})$, with $\mathbf{x}=(\ell^{k/\ell},0^{e_2-k/\ell})$ where $\ell=1+(k-1)(e_1-1)/(v-1)$, gives an example of a $2$-design with the required properties. 

\medskip
(b) \emph{Comment on Theorem~\ref{thm:ex-ft} for the case $s\geq3$.}\quad For each $s \geq 3$ we can obtain a flag-transitive, $(s-1)$-chain-imprimitive $2$-design from a flag-transitive, $s$-chain-imprimitive $2$-design, as follows: Given a point set $\PP$ with an $s$-chain $\Ch$ of partitions $\C_1, \ldots, \C_s$, with parameters $e_1, \ldots, e_s$ as in Theorem \ref{thm:ft}, we form a new chain $\Ch'$ of partitions $\C'_0, \ldots, \C'_{s-1}$ of $\PP$ by deleting any nontrivial partition $\C_i$ ($1 \leq i \leq s-1$) from $\Ch$, so that $\C'_j = \C_j$ for $0 \leq j \leq i-1$ and $\C'_j = \C_{j+1}$ for $i \leq j \leq s-1$. Let $e'_j := e_j$ for $0 \leq j \leq i-2$, $e'_{i-1} := e_ie_{i-1}$, and $e'_j := e_{j+1}$ for $i \leq j \leq s-1$. Then the partition $\C'_j$ contains $e'_j$ classes of $\C'_{j-1}$, for $1 \leq j \leq s-1$. Suppose that $\D = \big(\PP,B^G\big)$ is a $G$-flag-transitive $2$-design such that $G$ is $s$-chain-imprimitive on $\PP$, preserving the $s$-chain $\Ch$, where $B$ is a $k$-subset of $\PP$. Then $G \leq W:= S_{e_1} \wr \ldots \wr S_{e_s}$, the stabiliser of the $s$-chain $\Ch$. We first replace $\B=B^G$ by the possibly larger block set $B^W$ as explained in the paragraph after the statement of Theorem~\ref{thm:ft}, and note that $W$ leaves the $(s-1)$-chain $\Ch'$ invariant, so  $W$ is a subgroup of the stabiliser $H := S_{e'_1} \wr \ldots \wr S_{e'_{s-1}}$ of $\Ch'$.  Define $\D' := \big(\PP, B^H\big)$. Then $\D'$ is an $H$-flag-transitive $2$-design (see, for example, \cite[Proposition 1.1]{CP93}), and by definition,  $H$ is $(s-1)$-chain-imprimitive on $\PP$, preserving the $(s-1)$-chain $\Ch'$.

\medskip
(c) \emph{Comment on further examples.}\quad  Using Theorem \ref{thm:ft} and \textsc{Magma} \cite{magma}, we were able to find all parameter sets $(e_1, e_2, e_3; k)$ with $e_1, e_2, e_3 \leq 50$ which correspond to flag-transitive, $3$-chain-imprimitive $2$-designs with automorphism group $G = S_{e_1} \wr S_{e_2} \wr S_{e_3}$ and block size $k$. These are listed in Table \ref{tab:flag-tr}; the entries in boldface in Table \ref{tab:flag-tr} correspond to parameters of designs obtained using Construction \ref{constr:family}. We see from Table \ref{tab:flag-tr} that there are many more explicit feasible parameter sets apart from those used in Construction \ref{constr:family}, and this suggests that there may be other methods for construction of  infinite families of flag-transitive, $s$-chain-imprimitive $2$-designs.

\begin{problem}
Find more infinite families of flag-transitive, $s$-chain-imprimitive $2$-designs, with $s\geq3$. In particular, find more families with $s$ unbounded.
\end{problem}

\emph{Data availability statement.} No datasets were generated or analysed in the current study.

\emph{Conflicts of interest.} All authors have no conflicts of interest to disclose.

\section{Chain imprimitivity} \label{s:prelims}

\subsection{Chain structure on a set} \label{ss:chain}

We shall use the same notation and terminology as in \cite{chainspaper} for the labelling of the point set $\PP$ and the partition chain $\Ch$.

Given integers $s \geq 2$ and $e_1, \ldots, e_s \geq 2$. The set $\PP$ is the Cartesian product $\prod_{i=1}^s \mathbb{Z}_{e_i}$. 

For any $i \in \{0, \ldots, s\}$ and any $(\delta_j)_{j>i} \in \prod_{j>i} \mathbb{Z}_{e_j}$, we define the subset $C_{(\delta_j)_{j>i}}$ of $\PP$ as follows:
  \begin{equation} \label{part}
    C_{(\delta_j)_{j>i}} := \big\{ (\varepsilon_j)_{j=1}^s \in \PP \ \big| \ \varepsilon_j = \delta_j \ \text{for all} \ j > i \big\}. 
    \end{equation}
When $i=0$, $(\delta_j)_{j>0}=(\delta_1,\delta_2,\ldots,\delta_s)$ is a point of $\PP$ and  $C_{(\delta_j)_{j>i}}$ is the singleton set containing that point. When $i=s$,  $(\delta_j)_{j>s}=(\,)$ is the unique empty tuple, and  the membership condition \eqref{part} for $ C_{(\delta_j)_{j>s}}=C_{(\,)}$ is vacuous, and hence $C_{(\delta_j)_{j>s}}=\PP$. By convention  $\prod_{j>s} \mathbb{Z}_{e_j}$ is the singleton set $\{(\,)\}$. For any $i \in \{0, \ldots, s\}$, we then define the family $\C_i$ of subsets by
    \begin{equation} \label{partn}
    \C_i := \Big\{ C_{(\delta_j)_{j>i}} \ \Big| \ (\delta_j)_{j>i} \in \prod_{j>i} \mathbb{Z}_{e_j} \Big\}.
    \end{equation}
Then $\C_0$ and $\C_s$ are trivial partitions of $\PP$ consisting of the singleton subsets of $\PP$ and of the entire set $\PP$, respectively, while $\C_1, \ldots, \C_{s-1}$ are nontrivial partitions of $\PP$. For each $i \in \{0, \ldots, s-1\}$ the partition $\C_i$ is a proper refinement of $\C_{i+1}$, which we write as $\C_i \prec \C_{i+1}$. Thus, under the partial order $\prec$, the partitions $\C_0, \ldots, \C_s$ form a chain $\Ch$ as in \eqref{eq:chain}. Moreover each set $\PP$ and chain $\Ch$ with parameters $e_1,\dots,e_s$ can be labelled in this way. 

Each partition $\C_i$ has parameters $c_i$, $d_i$, and $e_i$, where $c_i$ is the size of each class in $\C_i$, $d_i$ is the number of classes in $\C_i$, $e_0 := 1$, and, for each $i \in \{1, \ldots, s\}$, $e_i$ is the number of $\C_{i-1}$-classes contained in each $\C_i$-class. Thus
    \begin{equation*} 
    c_i = \prod_{j=0}^i e_j \quad \text{and} \quad d_i = \prod_{j=i+1}^s e_j.
    \end{equation*}
In particular $c_0 = d_s = 1$, $c_s = d_0 = |\PP|$, $c_1 = e_1$, and $d_{s-1} = e_s$.

For $i \in \{0, \ldots, s\}$ and any nonempty subset $X \subseteq \PP$, we define
    \begin{equation} \label{C(X)}
    \C_i(X) := \{ C \in \C_i \ | \ X \cap C \neq \varnothing \}.
    \end{equation}
In particular, if $C \in \C_i$, then $\C_{i-1}(C)$ is the set of all $\C_{i-1}$-classes that are contained in $C$, and $\C_{i+1}(C) = \{C^+\}$, where 
    \begin{equation*} 
    C^+ := \text{the unique $\C_{i+1}$-class that contains $C \in \C_i$}.
    \end{equation*}
It follows that for any subset $X \subseteq \PP$ and any class $C \in \C_i$, the set of all $\C_{i-1}$-classes that are contained in $C$ and that contain points of $X$ is
    \begin{align*}
    \C_{i-1}(X) \cap \C_{i-1}(C)
    &= \{ C' \in \C_{i-1} \ | \ X \cap C' \neq \varnothing \ \text{and} \ C' \subseteq C \} \\
    &= \{ C' \in \C_{i-1} \ | \ C' \cap X \cap C \neq \varnothing \} \\
    &= \C_{i-1}(X \cap C).
    \end{align*}

\subsection{Iterated wreath product of groups} \label{subsec:wreath}

Using the notation of Subsection \ref{ss:chain}, for each $i \in \{1, \ldots, s\}$ let $G_i \leq \Sym(\mathbb{Z}_{e_i}) = S_{e_i}$. We denote by $G := G_1 \wr \ldots \wr G_s$ the \emph{iterated wreath product} of the groups $G_1, \ldots, G_s$. For details of the definition of the group $G$, its binary operation, and its actions on $\PP$ and on each of the partitions $\C_1, \ldots, \C_s$ we refer the reader to the treatment in \cite[Subsection 2.3]{chainspaper}, based on \cite{BPRS}.

The following fact will be useful in our proofs. Observe that for $s\geq2$, $1 \leq j \leq s-1$, and for each $C \in \C_j$,
    \begin{equation} \label{eq:chC}
    \Ch(C) : \ {\textstyle\binom{C}{1}} = \C_0(C) \prec \C_1(C) \prec \ldots \prec \C_{j-1}(C) \prec \C_j(C) = \{C\}
    \end{equation}
is a $j$-chain on $C$. For a group $H$ acting on a set $X$ we denote by $H^X$ the subgroup of $\Sym(X)$ induced by $H$.

\begin{proposition} \label{p:iterwr}
Let $s \geq 2$, $1 \leq i \leq s$, and $C \in \C_i$, and let $G = S_{e_1} \wr \ldots \wr S_{e_s}$.
    \begin{enumerate}[(a)]
    \item The stabiliser in $\Sym(C)$ of the chain $\Ch(C)$ in \eqref{eq:chC} is $H(C) := G_C^C \cong S_{e_1} \wr \ldots \wr S_{e_i}$, and $H(C)^{\C_{i-1}(C)} \cong S_{e_i}$.
    \item Further, for $1 \leq j \leq i-1$, $H(C) \cong \left( \prod_{C' \in \C_j(C)} H(C') \right) \rtimes H(C)^{\C_j(C)}$, where $H(C)^{\C_j(C)} \cong S_{e_{j+1}} \wr \ldots \wr S_{e_i}$ and $|\C_{j-1}(C)| = \prod_{\ell=j}^i e_\ell$.
    \item In particular, for $1 \leq j \leq s-1$, $G = \left( \prod_{C' \in \C_j} H(C') \right) \rtimes G^{\C_j}$, where $G^{\C_j} \cong S_{e_{j+1}} \wr \ldots \wr S_{e_s}$ and $|\C_j| = d_j = \prod_{i=j+1}^s e_i$.
    \end{enumerate}
\end{proposition}
\begin{proof}
(a) This follows from \cite[Theorem 2.1 (i)]{chainspaper} applied to $\C(C)$
, and from \cite[Theorem 2.1 (ii)]{chainspaper} applied to $H(C)^{\C_{j-1}(C)}$.

(b) By the discussion before \cite[Theorem 2.1]{chainspaper} applied to $H(C)$, the kernel $H(C)_{(\C_j(C))}$ of the $H(C)$-action on $\C_j(C)$ is $H(C)_{(\C_j(C))} = \prod_{C' \in \C_j} H(C') \cong (S_{e_1} \wr \ldots \wr S_{e_j})^{|\C_j(C)|}$, and the induced subgroup $H(C)^{\C_j(C)}$ on $\C_j(C)$ is $H(C)^{\C_j(C)} \cong H(C)/H(C)_{(\C_j(C))} \cong S_{e_{j+1}} \wr \ldots \wr S_{e_i}$, where $|\C_j(C)| = \prod_{\ell=j+1}^i e_\ell$. Thus $H(C) \cong H(C)_{(\C_j(C))} \rtimes H(C)^{\C_j(C)}$, and statement (b) follows.

(c) This follows by applying (b) to the case where $i=s$, since in this case $C = \PP$ and $H(C) = G$.
\end{proof}

\subsection{Array of a subset} \label{subsec:array}

The \emph{array function} of a nonempty subset $B$ of $\PP$, with respect to the chain $\Ch$ of partitions, is the function $\chi_B : \bigcup_{i=1}^s \C_i \rightarrow \mathbb{Z}_{\geq 0}$ defined by
    \begin{equation} \label{array}
    \chi_B(C) := \left| B \cap C \right| \quad \text{for each $C \in \C_i$ and each $i \in \{1, \ldots, s\}$.}
    \end{equation}
In particular, for $i=s$ we have $\chi_B(C) = \chi_B(\PP) = |B|$. For brevity, and when there is no ambiguity, we shall frequently use the notation $x_C := \chi_B(C)$, or, more specifically,
    \begin{equation}\label{def:x}
    x_{(\delta_j)_{j>i}} := \chi_B\big(C_{(\delta_j)_{j>i}}\big)
    \end{equation}
for $C_{(\delta_j)_{j>i}}$ as  in \eqref{part}. For a subgroup $G \leq S_{e_1} \wr \ldots \wr S_{e_s}$ which leaves invariant each partition in the chain $\Ch$, we say that two array functions $\chi_B$ and $\chi_{B'}$ are \emph{equivalent under $G$} if $\chi_B^g = \chi_{B'}$ for some $g \in G$, where
    \begin{equation} \label{chi^g}
    \chi_B^g(C) := \chi_B\big(C^{g^{-1}}\big) \quad \text{for all} \ C \in \bigcup_{i=1}^s \C_i. 
    \end{equation}
By \cite[Lemma 2.4]{chainspaper}, if $G = S_{e_1} \wr \ldots \wr S_{e_s}$ then the $G$-orbit of $B \subseteq \PP$ consists of all $B' \subseteq \PP$ whose array function $\chi_{B'}$ is equivalent to $\chi_B$ under $G$.

\subsection{Block-transitive, chain-imprimitive $2$-designs}

Consider a point-block incidence structure $\D = (\PP,\B)$, where $\PP = \prod_{i=1}^s \mathbb{Z}_{e_i}$ and $\B = B^G$, for some nonempty $B \subseteq \PP$ and with $G = S_{e_1} \wr \ldots \wr S_{e_s}$ leaving invariant an $s$-chain $\Ch$ of partitions $\C_i$ of $\PP$ as in \eqref{partn}. Then $\D$ is $G$-block-transitive and $(G,s)$-chain-imprimitive. Theorem 1.3 in \cite{chainspaper}, stated below, gives necessary and sufficient conditions on the array $\chi_B$ in order for $\D$ to be a $2$-design. Recall that for $C \in \C_{i-1}$, $C^+$ denotes the unique $\C_i$-class containing $C$.

\begin{theorem} \cite[Theorem 1.3]{chainspaper} \label{mainthm:2des}
Let $G = S_{e_1} \wr \ldots \wr S_{e_s}$ preserving $s$-chain $\Ch$ as in \eqref{partn}. Let $\D = (\PP,\B)$ where $\B = B^G$ for some $k$-subset $B$ of $\PP$. For any $C \in \bigcup_{i=1}^s \C_i$ let $x_C = |B \cap C|$. Then $\D$ is a $2$-design if and only if
    \begin{equation} \label{2des-1}
    \sum_{C \in \C_1} x_{C} \left( x_{C} - 1 \right)
    = \frac{k(k-1)}{v-1} (e_1 - 1)
    \end{equation}
and
    \begin{equation} \label{2des}
    \text{for each $i \in \{2, \ldots, s\}$,} \quad \sum_{C \in \C_{i-1}} x_C \left( x_{C^+} - x_C \right) = \frac{k(k-1)}{v-1} (e_i - 1) \prod_{j \leq i-1} e_j.
    \end{equation}
\end{theorem}

Now, $C_{(\delta_j)_{j>i-1}}^+ = C_{(\delta_j)_{j>i}}$ by \eqref{part}. Thus, using notation \eqref{def:x}, we can rewrite condition \eqref{2des} as
    \begin{equation*} 
    \sum_{(\delta_j)_{j>i-1} \in \prod_{j>i-1} \mathbb{Z}_{e_j}} x_{(\delta_j)_{j>i-1}} \left( x_{(\delta_j)_{j>i}} - x_{(\delta_j)_{j>i-1}} \right)
    = \frac{k(k-1)}{v-1} (e_i - 1) \prod_{j \leq i-1} e_j.
    \end{equation*}

\section{Point subsets with uniform array functions} \label{sec:arrays}

The theme of the paper is $G$-flag-transitive designs, where $G$ is chain-imprimitive on the point set $\PP$ with respect to a chain $\Ch$ as in \eqref{eq:chain}. If $B$ is a block of such a design then the setwise stabiliser $G_B$ is transitive on $B$. This means in particular that,   for each $i \in \{0, \ldots, s\}$, every $\C_i$-class $C$ which contains elements of $B$ must contain a constant number of elements of $B$. In other words, there exist positive integers $y_i$, depending only on $i$, such that for each $i \in \{0, \ldots, s\}$ and each $C \in \C_i$, $\chi_B(C)=  \big| B \cap C \big| = x_{C} \in \{0, y_i\}$.
\begin{definition} \label{unifseq}
A subset $B$ of $\PP$ is \emph{uniform} relative to $\Ch$, with \emph{uniform sequence} $(y_0, y_1, \ldots, y_s)$, if for each $i \in \{0, \ldots, s\}$ and each $C \in \C_i(B)$, $\big| B \cap C \big| = y_i$.
\end{definition}

Note that for $i=0$ the class $C$ has a unique element, say $\delta$, and $\chi_B(C) = 0$ or $1$ according to whether $\delta \in B$ or $\delta \notin B$, so $y_0=1$; while for $i=s$ we have $C = \PP$, and so $\chi_B(\PP) = |B| = k$ and thus $y_s = k$. Our next result derives further restrictions on a uniform sequence and shows that all uniform subsets with the same uniform sequence are in the same orbit under the stabiliser of the chain $\Ch$.


\begin{lemma} \label{lem:unif}
Let $s,e_1,\dots,e_s$ be integers, all at least $2$, let $\PP = \prod_{i=1}^s \mathbb{Z}_{e_i}$, let $\Ch$ be a chain of nontrivial partitions of $\PP$, as in \eqref{eq:chain} with parts as in \eqref{partn}, and chain stabiliser $G = S_{e_1} \wr \ldots \wr S_{e_s}$. For each $i \in \{0, \ldots, s\}$, let $y_i \in \mathbb{Z}^+$.
    \begin{enumerate}[(a)]
    \item There exists a nonempty subset $B$ of $\PP$ which is uniform relative to $\Ch$ with uniform sequence $(y_0, \ldots, y_s)$, if and only if the following hold: 
    \[
    \hbox{(i) $y_0 = 1$, $y_s=|B|$, \quad and \quad (ii) for all $i \in \{1, \ldots, s\}$, $y_{i-1} \mid y_i$ with $\frac{y_i}{y_{i-1}} \leq e_i$.}
    \]
    \item If conditions (i) and (ii) in part (a) hold, then the subset
        \begin{equation} \label{B-unif}
        B := \left\{ (\delta_i)_{i=1}^s \in \PP \ \ \vline \ \ 0 \leq \delta_i \leq \frac{y_i}{y_{i-1}} - 1 \ \text{for }\ 1 \leq i \leq s \right\}
        \end{equation}
    is uniform relative to $\Ch$ with uniform sequence $(y_0, \ldots, y_s)$.     Moreover, a subset of $\PP$ has these properties if and only if it lies in the $G$-orbit of $B$.
    \end{enumerate}
\end{lemma}

\begin{proof}
(a) Assume that there is a nonempty subset $B$ of $\PP$ such that $B$ is uniform relative to $\Ch$ with uniform sequence $(y_0, \ldots, y_s)$. Since all $\C_0$-classes have size $1$ and $B$ is non-empty, we have that $y_0=1$, and since $\C_s=\{\PP\}$ we also have $y_s=|B\cap \PP|=|B|$. This proves condition (i). Let $1\leq i\leq s$ and let $C$ be a $\C_i$-class intersecting $B$ non-trivially, so that $x_C=y_i$. Consider the set $\C_{i-1}(C)$ of all $\C_{i-1}$-classes contained in $C$. For each $C'\in \C_{i-1}(C)$, $x_{C'}\in\{0,y_{i-1}\}$, and $x_C=\sum_{C'\in \C_{i-1}(C)}x_{C'}$. It follows that  $y_{i-1}\mid y_i$. Recall that $|\C_{i-1}(C)| = e_i$ so, if $C'\subset C$ is such that $x_{C'}=y_{i-1}$, then we have  $x_C \leq e_i x_{C'}$ and thus $\frac{y_i}{y_{i-1}} = \frac{x_C}{x_{C'}} \leq e_i$, which is condition (ii). Therefore (i) and (ii) hold whenever some subset $B$ is uniform relative to $\Ch$ with uniform sequence $(y_0, \ldots, y_s)$.

Conversely assume that conditions (i) and (ii) hold. Note that from condition (ii), $0 \leq \frac{y_i}{y_{i-1}} - 1 < e_i$. Hence if $(\delta_i)_{i=1}^s$ satisfies the defining conditions for $B$ then $(\delta_i)_{i=1}^s \in \prod_{i=1}^s \mathbb{Z}_{e_i} = \PP$. The array function $\chi_B$ of $B$ with respect to $\Ch$ has value $\chi_B\big(C_{(\delta_j)_{j>i}}\big) = \big|\big\{ (\varepsilon_j)_{j=1}^s \in B \ \big| \ \varepsilon_j = \delta_j \ \text{for all} \ j > i \big\}\big|$, which, using conditions (i) and (ii), can be computed as
    \begin{equation} \label{arrayB'}
    \chi_B\big(C_{(\delta_j)_{j>i}}\big)
    = \begin{cases}
      \prod_{\ell=1}^i \frac{y_\ell}{y_{\ell-1}} = y_i &\text{if $\delta_j < \frac{y_j}{y_{j-1}}$ for all $j > i$}, \\
      0 &\text{otherwise}.
    \end{cases}
    \end{equation} 
Therefore $B$ is uniform relative to $\Ch$ with uniform sequence $(y_0, \ldots, y_s)$. This completes the proof of statement (a), and also proves the first part of statement (b).
(b) Let $\B = B^G$ with $B$ as in \eqref{B-unif}. Assume first that $B' \in \B$; so $B'=B^g$ for some $g\in G$. By \cite[Lemma 2.5 (a)]{chainspaper}, the array function $\chi_{B'} = \chi^g_B$ where $\chi^g_B$ is as defined in \eqref{chi^g}. Hence, for each $C \in \bigcup_{i=1}^s \C_i$, say $C\in\C_i$ so also $C^{g^{-1}}\in\C_i$, we have $\chi_{B'}(C)=\chi_{B}^g(C) = \chi_B\big(C^{g^{-1}}\big)$ which lies in $\{0,y_{i}\}$ since $B$ has uniform sequence $(y_0, \ldots, y_s)$. Thus $B'$ is also uniform relative to $\Ch$ with uniform sequence $(y_0, \ldots, y_s)$. 


Conversely, let $B'$ be a uniform subset of $\PP$ relative to $\Ch$ with uniform sequence $(y_0, \ldots, y_s)$; we will show that $B' \in \B$. 

\smallskip\noindent
\emph{Claim. For every $j \in \{0, \ldots, s-1\}$, there exists $g \in G$ such that $\C_j(B)^g = \C_j(B')$. }

We prove the claim by induction on $\ell = s-j$, for $\ell \in \{1, \ldots, s\}$. Suppose first that $\ell = 1$, so that $j=s-1$. By \eqref{C(X)}, $|\C_{s-1}(B)|$ is the number of classes of $\C_{s-1}$ that intersect $B$ nontrivially. Since $B$ and $B'$ both have uniform sequence $(y_0, \ldots, y_s)$, it follows from Definition~\ref{unifseq} and Lemma~\ref{lem:unif} that $|B|=|B'|=k=y_s$ and $|B\cap C|=y_{s-1}$ for $C\in \C_{s-1}(B)$, while  $|B'\cap C|=y_{s-1}$ for $C\in \C_{s-1}(B')$. Thus  $|\C_{s-1}(B)| = |\C_{s-1}(B')|= \frac{y_s}{y_{s-1}}$. Then since $G$  induces $G^{\C_{s-1}} \cong S_{e_s}$  on $\C_{s-1}$ (Proposition \ref{p:iterwr}(b)), there exists $g \in G$ such that  $g^{\C_{s-1}}$ maps $\C_{s-1}(B)$ to $\C_{s-1}(B')$. Therefore the claim holds for $j=s-1$, that is, for $\ell=1$.



Let $\ell \in \{2, \ldots, s\}$ and assume inductively that the claim holds for $\ell-1$, that is, there exists $g \in G$ such that $\C_{j}(B)^g = \C_{j}(B')$ where $j := s-\ell+1$. Note that $\C_j(B)^g = \C_j(B^g)$ by \eqref{C(X)}. We proved above that $B^g$ is uniform relative to $\Ch$ with uniform sequence $(y_0, \ldots, y_s)$. Thus for any $C \in \C_j(B^g) = \C_j(B')$ we have $|\C_{j-1}(B^g \cap C)| = |\C_{j-1}(B' \cap C)| = \frac{y_j}{y_{j-1}}$. By Proposition \ref{p:iterwr}(a), the group $H(C) = G_C^C \cong S_{e_1} \wr \ldots \wr S_{e_j}$, and $H(C)^{\C_{j-1}(C)} \cong S_{e_j}$. Hence there exists an element $h_C \in H(C)$ such that $\C_{j-1}(B^g \cap C)^{h_C}=\C_{j-1}(B' \cap C)$, and we may view $h_C$ as a permutation of $\PP$ (and hence an element of $G$) fixing pointwise each of the parts of $\C_j\setminus\{C\}$. Such an element $h_C$ exists for each of the parts $C \in \C_j(B^g)$, and the product of these elements $h_C$ over all $C \in \C_j(B^g)$ is an element of $G$ that maps $\C_{j-1}(B^g)$ (which is the union over $C \in \C_j(B^g)$ of $\C_{j-1}(B^g\cap C)$) to $\C_{j-1}(B')$. Since $j-1 = s-\ell$, the claim is proved for $\ell$.


Therefore the claim is proved by induction for all $\ell \in \{1, \ldots, s\}$.
In particular, for $\ell=s$, there exists $g \in G$ such that $\C_0(B^g) = \C_0(B')$. Since $\C_0(B^g) = \binom{B^g}{1}$ and $\C_0(B') = \binom{B'}{1}$,  it follows that $B^g = B'$. Therefore $B' \in \B$, which completes the proof of part (b).
\end{proof}

Recall from Section~\ref{subsec:array} that the array functions $\chi_B$ and $\chi_{B'}$ of two subsets $B, B'\subseteq \PP$ are equivalent under $G$ if $\chi_B^g = \chi_{B'}$ for some $g \in G$. Further, by \cite[Lemma 2.5(b)]{chainspaper}, if $G = S_{e_1} \wr \ldots \wr S_{e_s}$ then the $G$-orbit of $B \subseteq \PP$ consists of all $B' \subseteq \PP$ whose array function $\chi_{B'}$ is equivalent to $\chi_B$ under $G$. 

\begin{lemma} \label{lem:b}
Let $s,e_1,\dots,e_s$ be integers, all at least $2$, let $\PP = \prod_{i=1}^s \mathbb{Z}_{e_i}$, and let $\Ch$ be a chain of nontrivial partitions of $\PP$, as in \eqref{eq:chain} with chain stabiliser $G = S_{e_1} \wr \ldots \wr S_{e_s}$. Let $\B = B^G$ where $B \subseteq \PP$ is uniform relative to $\Ch$ with uniform sequence $(y_0, \ldots, y_s)$. Then
    \begin{equation} \label{eq:b}
    |\B| = \prod_{j=1}^s \binom{e_j}{\frac{y_j}{y_{j-1}}}^{k/y_j}.
    \end{equation}
\end{lemma}

\begin{proof}
As in Proposition \ref{p:iterwr}, we define the group $H(C) = G_C^C$ for any class $C \in \bigcup_{j=1}^s \C_j$. Note that if $C \in \C_i(B)$, then $B \cap C$ is uniform relative to $\C(C)$ with uniform sequence $(y_0, \ldots, y_i)$, since for each $j \in \{1, \ldots, i\}$ and each $C' \in \C_j(C)$, $\chi_{B \cap C}(C') = |B \cap C \cap C'| = |B \cap C'| = \chi_B(C') \in \{0,y_j\}$.

\smallskip\noindent
\emph{Claim 1. For any $i \in \{1, \ldots, s\}$ and any $C \in \C_i(B)$, the orbit $(B \cap C)^{H(C)}$ is the set of all subsets of $C$ that have uniform sequence $(y_0, \ldots, y_i)$ relative to $\C(C)$.} 

If $i=1$ then by Proposition \ref{p:iterwr}(a) the group $H(C) \cong S_{e_1}$. Hence $(B \cap C)^{H(C)}$ is the set of all subsets of $C$ that have size $y_1$, and thus have uniform sequence $(y_0,y_1)$ relative to $\C(C)$. Therefore the claim holds for $i=1$. Now assume that $i \geq 2$.
Using Lemma \ref{lem:unif}(b) (and by replacing $\PP$ in the lemma with the set $C$, the chain $\C$ with $\C(C)$, and the group $G$ with $H(C)$), we deduce that the orbit $(B \cap C)^{H(C)}$ consists of all subsets of $C$ that have size $y_i$ and are uniform relative to $\C(C)$, with uniform sequence $(y_0, \ldots, y_i)$. This proves Claim 1.

\smallskip
For $i \in \{1, \ldots, s\}$ let $b_i := \prod_{j=1}^i \binom{e_j}{y_j/y_{j-1}}^{y_i/y_j}$.

\smallskip\noindent
\emph{Claim 2. For any $i \in \{1, \ldots, s\}$ and any $C_i \in \C_i(B)$, $\big|(B \cap C_i)^{H(C_i)}\big| = b_i$.} 

We prove this by induction on $i$. If $i=1$ then by Claim 1 the orbit $(B \cap C_1)^{H(C_1)}$ is the set of all subsets of size $y_1$ of $C_1$; since $|C_1| = e_1$, the number of such subsets is $\binom{e_1}{y_1}$. Therefore, recalling that $y_0 = 1$, we have
    \[
    \big|(B \cap C_1)^{H(C_1)}\big| = \binom{e_1}{y_1} = \prod_{j=1}^1 \binom{e_1}{y_1/y_0}^{y_1/y_1} = b_1,
    \]
which shows that Claim 2 holds when $i=1$.

Suppose now that $i \geq 2$ and assume that $\big|(B \cap C_{i-1})^{H(C_{i-1})}\big| = b_{i-1}$ for any $C_{i-1} \in \C_{i-1}(B)$. By Claim 1, each subset $B' \in (B \cap C_i)^{H(C_i)}$ has size $y_i$, and is uniform relative to $\C(C_i)$ with uniform sequence $(y_0, \ldots, y_i)$. The set $B'$ is the disjoint union of subsets $B' \cap C_{i-1}$ for all $C_{i-1} \in \C_{i-1}(B')$, where, by the observation at the start of this proof, each subset $B' \cap C_{i-1}$ has size $y_{i-1}$ and is uniform relative to $\C(C_i)$ with uniform sequence $(y_0, \ldots, y_{i-1})$. Since $B' \cap C_{i-1} \subseteq C_{i-1}$, it follows from Claim 1 that $B' \cap C_{i-1} \in (B \cap C_{i-1})^{H(C_{i-1})}$. Conversely, note that any subset $B''$ from an orbit $(B \cap C_{i-1})^{H(C_{i-1})}$, for any $C_{i-1} \in \C_{i-1}(B \cap C_i)$, is uniform relative to $\C(C_i)$ with uniform sequence $(y_0, \ldots, y_{i-1})$. The number $|\C_{i-1}(B \cap C_i)|$ of $\C_{i-1}$-classes in $C_i$ that intersect $B$ non-trivially is $|\C_{i-1}(B \cap C_i)| = \frac{y_{i}}{y_{i-1}}$, so that $\big|\bigcup_{C_{i-1} \in \C_{i-1}(B \cap C_i)} B'' \big| = |B''| \cdot |\C_{i-1}(B \cap C_i)| = y_{i-1} \cdot y_i/y_{i-1}$. Thus, if we take one subset $B''$ from each of the orbits $(B \cap C_{i-1})^{H(C_{i-1})}$, where $C_{i-1} \in \C_{i-1}(B \cap C_i)$, then the union of these subsets $B''$ is a subset of $B \cap C_i$ of size $y_i$ that is uniform relative to $\C(C_i)$ with uniform sequence $(y_0, \ldots, y_i)$. By Proposition \ref{p:iterwr}(b) we deduce that
    \[
    H(C_i)
    \cong \left( \prod_{C \in \C_{i-1}(C_i)} H(C) \right) \rtimes H(C_i)^{\C_{i-1}(C_i)}
    \cong \left( \prod_{C \in \C_{i-1}(C_i)} H(C) \right) \rtimes S_{e_i}.
    \]
Since $B \cap C_i$ is the disjoint union over all $C \in \C_{i-1}(B \cap C_i)$ of subsets $B \cap C$, it follows that the orbit $(B \cap C_i)^{\prod_{C \in \C_{i-1}(C_i)} H(C)}$ is the disjoint union over all $C \in \C_{i-1}(B \cap C_i)$ of orbits $(B \cap C)^{H(C)}$. Hence
    \[
    (B \cap C_i)^{H(C_i)}
    = \left( \bigcup_{C \in \C_{i-1}(B \cap C_i)} (B \cap C)^{H(C)} \right)^{S_{e_i}}
    \]
and therefore, since $|\C_{i-1}(B \cap C_i)| = y_i/y_{i-1}$,
    \[
    \big| (B \cap C_i)^{H(C_i)} \big|
    = \big| (B \cap C)^{H(C)} \big|^{y_i/y_{i-1}} \binom{e_i}{y_i/y_{i-1}}
    \quad \text{for any $C \in \C_{i-1}(B \cap C_i)$.}
    \]
By the induction hypothesis $\big| (B \cap C)^{H(C)} \big| = b_{i-1}$. Substituting this into the equation above, and applying the definition of $b_{i-1}$, we obtain 
    \[
    \big| (B \cap C_i)^{H(C_i)} \big|
    = \Bigg(\prod_{j=1}^{i-1} \binom{e_j}{y_j/y_{j-1}}^{y_{i-1}/y_{j}}\Bigg)^ {y_i/y_{i-1}} \cdot \binom{e_i}{y_i/y_{i-1}}
    = \prod_{j=1}^i \binom{e_j}{y_j/y_{j-1}}^{y_i/y_j}
    = b_i.
    \]
Therefore Claim 2 holds by induction.

Finally, observe that $C_s = \PP$ and $G = H(C_s)$, so that $\B = B^G = (B \cap C_s)^{H(C_s)}$. Therefore, by Claim 2 and recalling that $y_s = k$, we obtain $|\B| = b_s = \prod_{j=1}^s \binom{e_j}{y_j/y_{j-1}}^{k/y_j}$. This completes the proof of the lemma.
\end{proof}

\begin{lemma} \label{lem:flagtr}
Let $s,e_1,\dots,e_s$ be integers, all at least $2$, let $\PP = \prod_{i=1}^s \mathbb{Z}_{e_i}$, and let $\Ch$ be a chain of nontrivial partitions of $\PP$, as in \eqref{eq:chain} with chain stabiliser $G = S_{e_1} \wr \ldots \wr S_{e_s}$. Then for any subset $B \subseteq \PP$ which is uniform relative to $\Ch$, 
the setwise stabiliser $G_{B}$ of $B$ is transitive on $B$.
\end{lemma}

\begin{proof}
Let $B \subseteq \PP$ be uniform relative to $\Ch$,  with uniform sequence $(y_0, \ldots, y_s)$. 
Since, by Lemma \ref{lem:unif}(b), every $G$-orbit on uniform subsets of $\PP$ with uniform sequence $(y_0, \ldots, y_s)$ contains a subset $B$ of the form described in \eqref{B-unif}, it is enough to prove the result for such a subset $B$. Observe that the subset $B$ in \eqref{B-unif}  can be written as $B = E_1 \times \ldots \times E_s$ where, for $i \in \{1, \ldots, s\}$,
    \[
    E_i := \left\{ \delta_i \in \mathbb{Z}_{e_i} \ \vline \ 0 \leq \delta_i < \frac{y_i}{y_{i-1}} \right\} = \mathbb{Z}_{y_i/y_{i-1}}.
    \]
For each $i \in \{1, \ldots, s\}$, the sets $B \cap C$, for all $C \in \C_i(B)$, form a partition $\C_i^B$ of $B$, and the partitions $\C_i^B$ form a chain $\Ch^B : \binom{B}{1} = \C_0^B \prec \C_1^B \prec \ldots \prec \C_s^B = \{B\}$. We note that some of the subsets $E_i$ may contain only a single element, (if $y_{i-1}=y_i$), and hence some of these partitions may be equal. This possible degeneracy does not affect our arguments. By Proposition \ref{p:iterwr}(a) the stabiliser in $\Sym(B)$ of this chain $\Ch^B$ is the group $L = \Sym(E_1) \wr \ldots \wr \Sym(E_s)$. Identify each group $\Sym(E_i)$ with the subgroup of $S_{e_i}$ acting naturally on $E_i$ and fixing $\mathbb{Z}_{e_i} \setminus E_i$ pointwise; thus $L \leq G$ and in particular $L \leq G_B$. We shall show that $L$ is transitive on $B$. By a result in permutation group theory (see, for instance, \cite[Lemma 5.4 (iii)]{PS}), the wreath product $X \wr Y$ of groups $X \leq \Sym(\Gamma)$ and $Y \leq \Sym(\Delta)$ is transitive in its natural action on $\Gamma \times \Delta$ if and only if $X$ is transitive on $\Gamma$ and $Y$ is transitive on $\Delta$. (This is true even if one or other of $\Gamma, \Delta$ has size $1$.) If $s=2$ then $L = \Sym(E_1) \wr \Sym(E_2)$, so by \cite[Lemma 5.4 (iii)]{PS} the group $L$ is transitive on $B$. Suppose that $s \geq 3$ and that $K = \Sym(E_1) \wr \ldots \wr \Sym(E_{s-1})$ is transitive on $E_1 \times \ldots \times E_{s-1}$; then $L = K \wr \Sym(E_s)$ and again by \cite[Lemma 5.4 (iii)]{PS} the group $L$ is transitive on $B = (E_1 \times \ldots \times E_{s-1}) \times E_s$. Therefore, by induction, $L$ is transitive on $B$. It follows that $G_B$ is transitive on $B$, completing the proof.
\end{proof}

A $1$-design is a point-block incidence structure in which every point lies in a constant number of blocks. Clearly, any point-block  incidence structure which admits a point-transitive automorphism group is a $1$-design.

\begin{corollary}\label{cor:flagtr}
Let $s,e_1,\dots,e_s$ be integers, all at least $2$, let $\PP = \prod_{i=1}^s \mathbb{Z}_{e_i}$, and let $\Ch$ be a chain of nontrivial partitions of $\PP$, as in \eqref{eq:chain} with chain stabiliser $G = S_{e_1} \wr \ldots \wr S_{e_s}$. Then for any subset $B \subseteq \PP$ which is uniform relative to $\Ch$ with uniform sequence $(y_0, \ldots, y_s)$, the point-block incidence structure $\D = (\PP,B^G)$ is a $G$-flag-transitive $1$-design.
\end{corollary}

\begin{proof}
The conclusion that $\D$ is a $1$-design follows from the transitivity of $G$ on $\PP$. Flag-transitivity follows from Lemma \ref{lem:flagtr}.
\end{proof}

\section{A general construction for flag-transitive, $s$-chain-imprimitive $2$-designs}\label{sec:gencon}

In this section we give a general construction for $2$-designs based on the existence of a sequence of integers $y_i$ as in \eqref{def:yi} with the properties  (FT1) and (FT2) of Theorem~\ref{thm:ft}. 
We first deduce from these properties a divisibility condition for these numbers $y_i$.

\begin{lemma}\label{lem:arith}
    Let $s,k,e_1,\dots,e_s$ be integers, all at least $2$ such that $k<v$ where $v=\prod_{i=1}^s e_i$. Let $d=\gcd(e_1-1,\dots,e_s-1)$, let $y_0=1$, and for $1\leq i\leq s$ let 
    $$ y_i = 1 + \frac{k-1}{v-1} \left( \left(\prod_{j=1}^i e_j\right) - 1 \right).$$
    Assume that $v-1$ divides $(k-1) \cdot d$.
        \begin{enumerate}
        \item[(a)] The number $u := \frac{(k-1)d}{v-1}\in\mathbb{Z}$ and each $y_i$ is a positive integer coprime to $u$. Moreover $y_1-y_0=\frac{k-1}{v-1}(e_1-1)$, and for $2\leq i\leq s$, $y_i-y_{i-1} = \frac{k-1}{v-1}(e_i-1)\prod_{j=1}^{i-1} e_j$. 
        \item[(b)] If, in addition, for each $i \in \{1, \ldots, s\}$ the integer $y_{i-1}$ divides $(e_i-1)\left(\prod_{j=1}^{i-1} e_j\right)/d$, then for each $i\in\{1,\dots,s\}$, $y_{i-1}$ divides $y_i$ and $1 < \frac{y_i}{y_{i-1}} < e_i$.
        \end{enumerate}
\end{lemma}

\begin{proof}
For $0\leq i\leq s$, let $c_i=\prod_{j=0}^ie_j$, with $e_0=1$. Then for each $i=1,\dots, s$,
\[
c_i-1 = c_{i-1}(e_i-1) +\dots +  c_{1}(e_{2}-1) + (e_1-1),
\]
and it follows that $d=\gcd(e_1-1,\dots,e_s-1)$ divides $c_i-1$; this holds also for $i=0$ since $c_0=1$.
By assumption $v-1$ divides $(k-1) \cdot d$, so  $u := \frac{(k-1)d}{v-1}\in\mathbb{Z}$. Thus we can write $y_i$ as
    \begin{equation}\label{y_i-u2}
    y_i = 1 + \frac{u}{d} \cdot (c_i - 1) \quad \text{for $0 \leq i \leq s$},
 \end{equation}
 proving that $y_i\in\mathbb{Z}^+$ and $y_i$ is coprime to $u$.   
Then, for $1\leq i\leq s$, using this equation, and noting that $c_i = c_{i-1}e_i$, we have
    \[ 
    y_i - y_{i-1} = \frac{u}{d} \cdot (c_i - c_{i-1}) = \frac{u}{d} \cdot (e_i - 1)c_{i-1}=
        \frac{k-1}{v-1}(e_i-1)\prod_{j=1}^{i-1} e_j. 
    \]
This proves part (a), noting that the product is `empty' if $i=1$. 

Now assume, in addition, that for each $i \in \{1, \ldots, s\}$ the integer 
$y_{i-1}$ divides $(e_i - 1)\left(\prod_{j=1}^{i-1} e_j\right)/d=\frac{(e_i - 1)c_{i-1}}{d}$. Thus, since $u \in \mathbb{Z}$, it follows that $y_{i-1}$ divides $y_i - y_{i-1}$, and therefore $y_{i-1}$ divides $y_i$.
Since $e_i\geq 2$ for all $i$, the $c_i$'s are strictly increasing. Then by \eqref{y_i-u2}, the $y_i$'s are also strictly increasing, and so $\frac{y_i}{y_{i-1}}>1$. 
Since $y_{i-1} = 1 + u(c_{i-1}-1)/d$,  by \eqref{y_i-u2}, we get that
    \[
    y_{i-1} e_i
    = e_i + u \cdot \frac{(c_{i-1} - 1)e_i}{d}
    = e_i \left( 1 - \frac{u}{d} \right) + u \cdot \frac{c_{i-1} e_i}{d}
    = e_i \left( 1 - \frac{u}{d} \right) + u \cdot \frac{c_{i}}{d}.
    \]
Now $k < v$ implies that $\frac{u}{d} = \frac{k-1}{v-1} < 1$, so $1 - \frac{u}{d} > 0$. Then, since $e_i \geq 2$, 
    \[
    e_i \left( 1 - \frac{u}{d} \right) + u \cdot \frac{c_{i}}{d}
    > \left( 1 - \frac{u}{d} \right) + u \cdot \frac{c_{i}}{d} = 1 + u \cdot \frac{c_i - 1}{d} = y_i.
    \]
Therefore  $\frac{y_i}{y_{i-1}} < e_i$. This proves part (b). 
\end{proof}

Now we give the general design construction.

\begin{construction} \label{constr:gen}
Let $s \geq 2$ and $e_1, \ldots, e_s \geq 2$ be integers, $v := \prod_{i=1}^s e_i$, $d := \gcd(e_1-1, \ldots, e_s-1)$, and let $k$ be an integer with $1 < k < v$. Let $y_0:=1$ and for $i \in \{1, \ldots, s\}$ let $y_i$ be as in \eqref{def:yi}, that is,
    \[ 
    y_i := 1 + \frac{k-1}{v-1} \left( \left( \prod_{1\leq j \leq i} e_j \right) - 1 \right) \quad \text{for $i \in \{1, \ldots, s\}$}.
    \]
Assume that properties (FT1) and (FT2) in Theorem \ref{thm:ft} both hold, that is, $v-1$ divides $(k-1) \cdot d$, and for each $i \in \{1, \ldots, s-1\}$, the integer $y_{i}$ divides $(e_{i+1} - 1) \left(\prod_{j=1}^{i} e_j\right)/d$. Let $\PP = \prod_{i=1}^s \mathbb{Z}_{e_i}$, and let $G = S_{e_1} \wr \ldots \wr S_{e_s}$, the stabiliser of the partition chain $\Ch$ as in \eqref{eq:chain} and \eqref{partn}. Also let $\mathbf{e} := (e_1, \ldots, e_s)$ and let  $\Dft(\mathbf{e};k)=(\PP,B^G)$, where 
$B$ is as in \eqref{B-unif}, that is, 
    \[
    B = \bigg\{ (\delta_1, \ldots, \delta_s) \ \bigg| \ 0 \leq \delta_i \leq \frac{y_i}{y_{i-1}} - 1 \ \text{for $1 \leq i \leq s$} \bigg\}.
    \]
\end{construction}

\begin{proposition} \label{prop:gen-ft}
With the hypotheses of Construction $\ref{constr:gen}$, 
\begin{enumerate}
    \item[(a)] the set $B$ has cardinality $k$, and
    \item[(b)] $\Dft(\mathbf{e};k)$ is a $G$-flag-transitive, $(G,s)$-chain-imprimitive $2$-$(v,k,\lambda)$ design, with $\lambda = \frac{bk(k-1)}{v(v-1)}$ where $b$ is as given in \eqref{eq:b} in Lemma \emph{\ref{lem:b}}.
\end{enumerate}  
\end{proposition}

\begin{proof}
By Lemma \ref{lem:arith}, for each $i \in \{1, \ldots, s\}$, $y_i$ is an integer, $y_{i-1}$ divides $y_i$, and $1 < \frac{y_i}{y_{i-1}} \leq e_i$.
Therefore the set $B$  is uniform relative to $\C$ with uniform sequence $(y_0,\ldots, y_s)$ by Lemma \ref{lem:unif}. In particular  $|B|=y_s=k$.

By Corollary~\ref{cor:flagtr}, $\D = \Dft(\mathbf{e};k)$ is a $G$-flag-transitive $1$-design, and by construction $\D$ is $(G,s)$-chain-imprimitive.  
It remains to prove that $\D$ is a $2$-design, and we do this by  verifying conditions \eqref{2des-1} and \eqref{2des} of Theorem~\ref{mainthm:2des} for $i \in \{2, \ldots, s\}$. 
%
%
Since $B$ has  uniform sequence $(y_0,\ldots, y_s)$, we have that $x_C = y_i$ for $C \in \C_i(B)$, and $x_C = 0$ for $C \in \C_i\setminus \C_i(B)$. The set $B$ is the disjoint union of intersections $B \cap C$ for all $C \in \C_i(B)$, so that $\sum_{C \in \C_i(B)} y_i = \sum_{C \in \C_i(B)} |B \cap C| = |B| = k$. In particular, for $i=1$,
    \[ 
    \sum_{C \in \C_1} x_C (x_C - 1) = \sum_{C \in \C_1(B)} y_1 (y_1 - 1) = \left(\sum_{C \in \C_1(B)} y_1 \right) (y_1 - 1) = k(y_1 - 1). 
    \]
Condition \eqref{2des-1} follows since, by the definition of $y_1$, we have $y_1 - 1 = \frac{k-1}{v-1}(e_1 - 1)$. Let $i \in \{2, \ldots, s\}$ and let $C^+$ be the unique $\C_i$-class containing $C \in \C_{i-1}$. As noted above, for $C \in \C_{i-1}(B)$ we have $x_C=y_{i-1}$ and $x_{C^+} = y_i$, and for $C \not\in \C_{i-1}(B)$ we have $x_C=0$. Thus
    \[ \sum_{C \in \C_{i-1}} x_{C} \left( x_{C^+} - x_{C} \right)
    = \sum_{C \in \C_{i-1}(B)} y_{i-1} \left( y_i - y_{i-1} \right)
    = \left(\sum_{C \in \C_{i-1}(B)} y_{i-1}\right)\cdot \left( y_i - y_{i-1} \right)
    = k(y_i - y_{i-1}). \]
    By Lemma~\ref{lem:arith}(a), $y_i - y_{i-1}= \frac{k-1}{v-1} (e_i - 1) \prod_{j=1}^{i-1} e_j$, and condition \eqref{2des} follows. Therefore, by Theorem \ref{mainthm:2des}, $\D = (\PP,\B)$ is a  $2$-design. The formula for $\lambda$ follows immediately since we determined $b=|B^G|$ in Lemma \emph{\ref{lem:b}}, completing the proof.
\end{proof}

\section{Proof of Theorem \ref{thm:ft} \label{sec:flagtr}}

Let $s, k, e_1, \ldots, e_s$ be integers, all at least $2$, such that $k < v$, where $v := \prod_{i=1}^s e_i$.
Let $\PP = \prod_{i=1}^s \mathbb{Z}_{e_i}$, of size $v$, and let $W = S_{e_1} \wr \ldots \wr S_{e_s}$, the stabiliser of the partition chain $\Ch$ as in \eqref{eq:chain} and \eqref{partn}. So for $1 \leq i \leq s$ each class of $\C_i$ contains $e_i$ classes of $\C_{i-1}$. Let $d := \gcd(e_1 - 1, \ldots, e_s - 1)$.

Suppose first that $y_1,\dots,y_s$ are as in \eqref{def:yi} and that conditions \eqref{2des-1} and \eqref{2des} of Theorem~\ref{mainthm:2des} hold. Then by Proposition~\ref{prop:gen-ft}, the design $\Dft(\mathbf{e};k)$ of Construction~\ref{constr:gen} is a $(G,s)$-chain-imprimitive, $G$-flag-transitive $2$-design, with blocks of size $k$, for the group $G=W$.

Now suppose conversely that $\D=(\PP,\B)$ is a $(G,s)$-chain-imprimitive, $G$-flag-transitive $2$-design, with blocks of size $k$, for some group $G\leq W$ (so $G$ leaves invariant the chain $\Ch$). As discussed after the statement of Theorem~\ref{thm:ft}, $(\PP, B^W)$ is also a $2$-design, for $B\in\B$, so in order to verify that conditions (FT1) and (FT2)  hold for the constants $y_i$ as in \eqref{def:yi}, we may assume that $G=W$. Since $\D$ is $G$-flag-transitive,  the setwise stabiliser $G_{B}$ of the block $B$ is transitive on $B$. Thus for each $i \in \{0, \ldots, s\}$, since $G_B$ leaves $\C_i$ invariant,  every $\C_i$-class which contains elements of $B$ must contain a constant number of elements of $B$.
That is, $B$ is uniform relative to $\Ch$ for some uniform sequence  $(y'_0, \ldots, y'_s)$.

\medskip\noindent\emph{Claim 1:  $(y'_0, \ldots, y'_s) = (y_0, \ldots, y_s)$. }

By \eqref{def:yi} and Lemma~\ref{lem:unif}(a), we have $y'_0 = 1=y_0$ and $y'_s = k=y_s$. So assume that $1\leq i<s$. Since $B$ is the disjoint union of all the intersections $B \cap C$ with $C \in \C_i(B)$, we have $\sum_{C \in \C_i(B)} y'_i = \sum_{C \in \C_i(B)} |B \cap C| = |B| = k$. Also, for any $C \in \C_i$, we have $x_C = y'_i$ if $C \in \C_i(B)$, and $x_C = 0$ otherwise; and since $\D$ is a $2$-design, it follows from Theorem~\ref{mainthm:2des} that conditions \eqref{2des-1} and \eqref{2des} in that theorem hold. We will show by induction that $y'_i=y_i$ for each $i\in\{1,\dots,s\}$.

By \eqref{2des-1} we have $\sum_{C \in \C_1} x_C (x_C - 1) = \frac{k(k-1)}{v-1} (e_1 - 1)$, while 
    \[ \sum_{C \in \C_1} x_C (x_C - 1) = \sum_{C \in \C_1(B)} y'_1 (y'_1 - 1) = k(y'_1 - 1). \]
Hence, $y'_1 = 1 + \frac{k-1}{v-1} (e_1 - 1)$ which equals $y_1$ by \eqref{def:yi}. Suppose now that $2\leq i\leq s$ and $y'_{i-1}=y_{i-1}$.  Note that $C \in \C_{i-1}(B)$ implies that $C^+ \in \C_i(B)$, where $C^+$ is the unique $\C_i$ class containing $C$. Thus, for any $C \in \C_{i-1}$, $x_C = y'_{i-1}$ and $x_{C^+} = y'_i$ whenever $C \in \C_{i-1}(B)$, and $x_C = 0$ otherwise, so
    \[ \sum_{C \in \C_{i-1}} x_C (x_{C^+} - x_C) = \sum_{C \in \C_{i-1}(B)} y'_{i-1} (y'_i - y'_{i-1}) = \left(\sum_{C \in \C_{i-1}(B)} y'_{i-1}\right) (y'_i - y'_{i-1})  = k(y'_i - y'_{i-1}). \]
On the other hand, by \eqref{2des}, $ \sum_{C \in \C_{i-1}} x_C (x_{C^+} - x_C) =\frac{k(k-1)}{v-1} (e_i - 1)\prod_{j=1}^{i-1} e_j$. Hence 
    \begin{equation} \label{y'_i}
    y'_i = y'_{i-1} + \frac{k-1}{v-1} (e_i - 1) \prod_{j \leq i-1} e_j.
    \end{equation}
By the induction hypothesis, we have
    \begin{align*}
    y'_i
    &= y_{i-1} + \frac{k-1}{v-1} (e_i - 1) \prod_{j \leq i-1} e_j \\
    &= 1 + \frac{k-1}{v-1} \left( \Bigg(\prod_{j \leq i-1} e_j\Bigg) - 1  +  (e_i - 1) \prod_{j \leq i-1} e_j\right) \\
    &= 1 + \frac{k-1}{v-1} \left( \Bigg(\prod_{j \leq i} e_j \Bigg) - 1 \right) = y_i. 
    \end{align*}
Thus Claim 1 is proved.

\medskip\noindent\emph{Claim 2: Condition {\rm (FT1)} of Theorem~{\rm\ref{thm:ft}} holds.}\quad
Observe that in equation \eqref{y'_i} (which holds for all $i\in\{1,\dots,s\}$) the numbers $y'_i, y'_{i-1} \in \mathbb{Z}$, so also $\frac{k-1}{v-1} (e_i - 1) \prod_{j \leq i-1} e_j \in \mathbb{Z}$. Now $v = \prod_{j=1}^s e_j$ so $\gcd\big( v-1, \ \prod_{j \leq i-1} e_j \big) = 1$. Hence $v-1$ divides $(k-1)(e_i - 1)$, and this is true for all $i \in \{1, \ldots, s\}$. Therefore $v-1$ divides $(k-1) \cdot d$ where $d =\gcd(e_1 - 1, \ldots, e_s - 1)$, proving (FT1). 

\medskip\noindent\emph{Claim 3: Condition {\rm (FT2)} of Theorem~{\rm\ref{thm:ft}} holds.}\quad 
Let $i\in\{1,\dots, s-1\}$,  and let $u := \frac{(k-1)d}{v-1}$, so $u\in\mathbb{Z}^+$ by Claim 2.
By Claims 1 and 2, the hypotheses of Lemma~\ref{lem:arith}(a) hold, and this result implies that all the $y_j$ are positive integers coprime to $u$, and that 
    \[ 
    y_{i+1} - y_i =  u \cdot \frac{e_{i+1} - 1}{d}\cdot \prod_{j \leq i} e_j . 
    \]
Since $B$ has uniform sequence  $(y_0, \ldots, y_s)$ relative to $\Ch$ by Claim 1,
the parameter $y_i$ divides $y_{i+1}$ by Lemma \ref{lem:unif}(a), and hence $y_i$ divides $y_{i+1} - y_i$. So $y_i$ divides $u \cdot \frac{e_{i+1} - 1}{d}\cdot \prod_{j \leq i} e_j$, and since $y_i$ is coprime to $u$,   $y_i$ divides $\frac{e_{i+1} - 1}{d}\cdot \prod_{j \leq i} e_j$. This proves condition (FT2), so Claim 3 is proved.


\medskip
This proves the main assertion of Theorem~\ref{thm:ft}, and the final assertion follows from Claim 1.
Thus Theorem~\ref{thm:ft} is proved.

\begin{remark}\label{rem:unique}
The discussion above allows us to conclude that, for given $s, e_1, \dots, e_s, v, k$, such that conditions (FT1) and (FT2) of Theorem~\ref{thm:ft} hold for the constants $y_i$ in \eqref{def:yi}, there is a \emph{unique}  $G$-flag-transitive $(G,s)$-chain-imprimitive $2$-design with block size $k$, for $G$ the full stabiliser of the $s$-chain $\Ch$ in \eqref{eq:chain}, namely the design $\Dft(\mathbf{e}; k)$ in Construction~\ref{constr:gen}. To see this note that, if $\D=(\PP,B^G)$ is such a design, then Claim 1 above shows that $B$ is uniform relative to $\Ch$ with uniform sequence $(y_0, \ldots, y_s)$, where each $y_i$ is as defined in \eqref{def:yi}. It follows from Lemma \ref{lem:unif}(b) that $\B$ consists of all subsets of $\PP$ which are uniform relative to $\Ch$ with uniform sequence $(y_0, \ldots, y_s)$. In particular, $\D = \Dft(\mathbf{e};y_s) = \Dft(\mathbf{e};k)$ is the unique such $2$-design.
\end{remark}

\section{Explicit constructions for flag-transitive, $s$-chain-imprimitive $2$-designs}\label{sec:expl}

Note that, although Construction \ref{constr:gen} gives a general construction for flag-transitive, $s$-chain-imprimitive $2$-designs based on a collection of integers $y_i$ satisfying certain conditions, we need to produce integers $y_i$ with the required conditions before we can guarantee existence of such $2$-designs. We do this in Construction \ref{constr:family}, where we give explicit parameters $e_1, \ldots, e_s$, and $k$ such that the integers $y_i$ satisfy the hypotheses of Construction \ref{constr:gen}. 

\begin{construction} \label{constr:family} 
Let $s, d$ be integers, both at least $2$, and define integers $e_1, \ldots, e_s, k$ by
    \begin{equation} \label{e_i}
    e_1=d+1,\quad     e_i = d + \prod_{j \leq i-1} e_j, \quad \text{for $2\leq i\leq s$, and}\quad  k:= 1 + \frac{v-1}{d} \text{ where }v=\prod_{j \leq s} e_j.
    \end{equation}
Let $\PP:=\prod_{i=1}^s\mathbb{Z}_{e_i}$, and let $G = S_{e_1} \wr \ldots \wr S_{e_s}$, the stabiliser of the partition chain $\Ch$ as in \eqref{eq:chain} with partitions as in \eqref{partn}. For $i\in\{1,\dots,s\}$, let $y_i$ be as in \eqref{def:yi}, let $B$ be as in \eqref{B-unif}, and let $\Dft(\mathbf{e};k)$ be as in Construction~\ref{constr:gen}, where $\mathbf{e} = (e_1, \ldots, e_s)$.
\end{construction}

\begin{proposition} \label{prop:ft}
The incidence structure $\Dft(\mathbf{e};k)$ of Construction \emph{\ref{constr:family}} is a $G$-flag-transitive, $(G,s)$-chain-imprimitive $2$-$(v,k,\lambda)$ design with $G = S_{e_1} \wr \ldots \wr S_{e_s}$ and $\lambda = \frac{bk}{vd}$, where $b = \prod_{i=1}^s \binom{e_i} {y_{i}/y_{i-1}}^{k/y_{i}}$ with $y_i=\frac{e_{i+1}-1}{d}$ for all $i \in \{1, \ldots, s-1\}$, and $y_s=k$.
\end{proposition}

\begin{proof}
First we prove that $d$ is equal to $d':=\gcd(e_1-1,\dots,e_s-1)$.  We show inductively that $d$ divides $e_i-1$ for each $i$: by \eqref{e_i} we have $d=e_1-1$ so this holds for $i=1$, and if $1<i\leq s$ and $d$ divides $e_j-1$ for all $j<i$, then by \eqref{e_i}, $e_i\equiv d+1\equiv 1\pmod{d}$, so $d$ also divides $e_i-1$. Thus $d$ divides $d'$. On the other hand $d'\leq e_1-1=d$ and hence $d=d'$. 

Thus to prove that $\Dft(\mathbf{e};k)$ is a $2$-design, it is sufficient to verify that properties (FT1) and (FT2) in Theorem \ref{thm:ft} hold, since then all the assertions will follow from Proposition~\ref{prop:gen-ft}. First we note that, by \eqref{e_i}, $v-1=(k-1)\cdot d$, and hence $v-1$ divides $(k-1)\cdot d$, proving property (FT1).  The second property (FT2) involves the $y_i$ for $1\leq i<s$.  Since $\frac{k-1}{v-1} = \frac{1}{d}$ we have, for each $i \in \{1, \ldots, s-1\}$, that $\prod_{j \leq i} e_j = e_{i+1} - d$, so by \eqref{def:yi},
    \begin{align*}
    y_i
    = 1 + \frac{k-1}{v-1} \left( \left( \prod_{j \leq i} e_j \right) - 1 \right)
    = 1 + \frac{1}{d} \left( \left( e_{i+1} - d \right) - 1 \right)
    = \frac{d + (e_{i+1} - d - 1)}{d}
    = \frac{e_{i+1} - 1}{d}.
    \end{align*}
This implies firstly that $y_i$ is a positive integer, since $d$ divides $e_{i+1}-1$. Also it implies that $y_i$ divides $(e_{i+1} - 1)\left(\prod_{j \leq i} e_j\right)/d$. Thus property (FT2) holds, and therefore $\Dft(\mathbf{e},k)$ is a $2$-design.

Recall that $\lambda = \frac{bk(k-1)}{v(v-1)}=\frac{bk}{vd}$, where $b := |\B|$ is computed using the formula \eqref{eq:b} in Lemma \ref{lem:b}. Recall $y_i = \frac{e_{i+1} - 1}{d}$ for $i<s$ and $y_s=k$.
This completes the proof.
\end{proof}

Having proved Proposition~\ref{prop:ft}, it is now simple to deduce Theorem~\ref{thm:ex-ft}.

\begin{proof}[Proof of Theorem \ref{thm:ex-ft}]
For each $s\geq 2$ we have infinitely many choices for the integer $d\geq2$, and hence, by Proposition~\ref{prop:ft}, Construction \ref{constr:family} yields infinitely many flag-transitive, $s$-chain-imprimitive $2$-designs, proving  Theorem \ref{thm:ex-ft}.
\end{proof}

\subsection{Further examples}
For the case where $s = 3$, by Theorem~\ref{thm:ft}, a flag-transitive, $3$-chain-imprimitive $2$-design $\Dft(\mathbf{e};k)$ as in Construction \ref{constr:gen} exists if and only if the parameters $\mathbf{e}=(e_1, e_2, e_3),$ and $k$ satisfy the following conditions:
    \begin{enumerate}
    \item $v-1 = e_1 e_2 e_3-1$ divides $(k-1) \cdot d$, where $d = \gcd(e_1 - 1, e_2 - 1, e_3 - 1)$;
    \item $1 + \frac{k-1}{e_1 e_2 e_3 - 1} (e_1 - 1)$ divides $\frac{(e_2 - 1)e_1}{d}$; 
    \item $1 + \frac{k-1}{e_1 e_2 e_3 - 1} (e_1 e_2 - 1)$ divides $\frac{(e_3 - 1)e_1 e_2}{d}$.
    \end{enumerate}
For each solution $\mathbf{e}=(e_1, e_2, e_3), k$, we obtain an example $\Dft(\mathbf{e};k)$ from the general construction in Construction~\ref{constr:gen}. An exhaustive search using \textsc{Magma} \cite{magma} produced all possible parameters satisfying these three conditions for which $e_1, e_2, e_3 \leq 50$. These are listed in Table \ref{tab:flag-tr} and we indicate in boldface font the parameter sets which are obtained using Construction \ref{constr:family}. Thus we have:

\begin{theorem} \label{thm:flagtr-3ch}
For  $e_1, e_2, e_3 \leq 50$ and $\mathbf{e}:=(e_1, e_2, e_3),$ there exists $k\in\mathbb{Z}$ such that the conditions of Construction~\emph{\ref{constr:gen}} are satisfied and $\Dft(\mathbf{e};k)$ is a flag-transitive $3$-chain-imprimitive $2$-design  if and only if $(e_1, e_2, e_3, k)$ are as listed in Table \ref{tab:flag-tr}.
\end{theorem}

\begin{table}[ht]
    \centering
    \begin{tabular}{rrrrrrr}
    \hline
    $e_1$ & $e_2$ & $e_3$ & $v = e_1 e_2 e_3$ & $k$ & $y_1$ & $y_2$ \\
    \hline\hline
    $\mathbf{3}$ & $\mathbf{5}$ & $\mathbf{17}$ & $\mathbf{255}$ & $\mathbf{128}$ & $\mathbf{2}$ & $\mathbf{8}$ \\
    $3$ & $5$ & $33$ & $495$ & $248$ & $2$ & $8$ \\
    $3$ & $5$ & $49$ & $735$ & $368$ & $2$ & $8$ \\
    $3$ & $9$ & $29$ & $783$ & $392$ & $2$ & $14$ \\
    $3$ & $13$ & $41$ & $1599$ & $800$ & $2$ & $20$ \\
    $4$ & $4$ & $10$ & $160$ & $54$ & $2$ & $6$ \\
    $4$ & $4$ & $19$ & $304$ & $102$ & $2$ & $6$ \\
    $4$ & $4$ & $28$ & $448$ & $150$ & $2$ & $6$ \\
    $4$ & $4$ & $37$ & $592$ & $198$ & $2$ & $6$ \\
    $4$ & $4$ & $46$ & $736$ & $246$ & $2$ & $6$ \\
    $4$ & $7$ & $16$ & $448$ & $150$ & $2$ & $10$ \\
    $\mathbf{4}$ & $\mathbf{7}$ & $\mathbf{31}$ & $\mathbf{868}$ & $\mathbf{290}$ & $\mathbf{2}$ & $\mathbf{10}$ \\
    $4$ & $7$ & $46$ & $1288$ & $430$ & $2$ & $10$ \\
    $4$ & $10$ & $22$ & $880$ & $294$ & $2$ & $14$ \\
    $4$ & $10$ & $43$ & $1720$ & $574$ & $2$ & $14$ \\
    $4$ & $13$ & $28$ & $1456$ & $486$ & $2$ & $18$ \\
    $4$ & $16$ & $34$ & $2176$ & $726$ & $2$ & $22$ \\
    $4$ & $19$ & $40$ & $3040$ & $1014$ & $2$ & $26$ \\
    $4$ & $22$ & $46$ & $4048$ & $1350$ & $2$ & $30$ \\
    $5$ & $7$ & $37$ & $1295$ & $648$ & $3$ & $18$ \\
    $5$ & $9$ & $17$ & $765$ & $192$ & $2$ & $12$ \\
    $5$ & $9$ & $33$ & $1485$ & $372$ & $2$ & $12$ \\
    $\mathbf{5}$ & $\mathbf{9}$ & $\mathbf{49}$ & $\mathbf{2205}$ & $\mathbf{552}$ & $\mathbf{2}$ & $\mathbf{12}$ \\
    $6$ & $6$ & $11$ & $396$ & $80$ & $2$ & $8$ \\
    $6$ & $6$ & $21$ & $756$ & $152$ & $2$ & $8$ \\
    $6$ & $6$ & $26$ & $936$ & $375$ & $3$ & $15$ \\
    $6$ & $6$ & $31$ & $1116$ & $224$ & $2$ & $8$ \\
    $6$ & $6$ & $41$ & $1476$ & $296$ & $2$ & $8$ \\
    $6$ & $11$ & $36$ & $2376$ & $476$ & $2$ & $14$ \\
    \hline
    \end{tabular} \hspace{0.5cm}
    \begin{tabular}{rrrrrrr}
    \hline
    $e_1$ & $e_2$ & $e_3$ & $v = e_1 e_2 e_3$ & $k$ & $y_1$ & $y_2$ \\
    \hline\hline
    $6$ & $11$ & $46$ & $3036$ & $1215$ & $3$ & $27$ \\
    $6$ & $16$ & $26$ & $2496$ & $500$ & $2$ & $20$ \\
    $6$ & $26$ & $41$ & $6396$ & $1280$ & $2$ & $32$ \\
    $7$ & $10$ & $37$ & $2590$ & $864$ & $3$ & $24$ \\
    $7$ & $25$ & $37$ & $6475$ & $1080$ & $2$ & $30$ \\
    $8$ & $8$ & $36$ & $2304$ & $330$ & $2$ & $10$ \\
    $8$ & $8$ & $50$ & $3200$ & $1372$ & $4$ & $28$ \\
    $8$ & $15$ & $22$ & $2640$ & $378$ & $2$ & $18$ \\
    $8$ & $15$ & $43$ & $5160$ & $738$ & $2$ & $18$ \\
    $8$ & $36$ & $50$ & $14400$ & $2058$ & $2$ & $42$ \\
    $9$ & $5$ & $17$ & $765$ & $192$ & $3$ & $12$ \\
    $9$ & $5$ & $33$ & $1485$ & $372$ & $3$ & $12$ \\
    $9$ & $5$ & $49$ & $2205$ & $552$ & $3$ & $12$ \\
    $9$ & $9$ & $29$ & $2349$ & $588$ & $3$ & $21$ \\
    $9$ & $13$ & $41$ & $4797$ & $1200$ & $3$ & $30$ \\
    $10$ & $7$ & $37$ & $2590$ & $864$ & $4$ & $24$ \\
    $10$ & $10$ & $28$ & $2800$ & $312$ & $2$ & $12$ \\
    $10$ & $28$ & $37$ & $10360$ & $1152$ & $2$ & $32$ \\
    $11$ & $16$ & $46$ & $8096$ & $1620$ & $3$ & $36$ \\
    $12$ & $12$ & $34$ & $4896$ & $891$ & $3$ & $27$ \\
    $15$ & $8$ & $22$ & $2640$ & $378$ & $3$ & $18$ \\
    $15$ & $8$ & $43$ & $5160$ & $738$ & $3$ & $18$ \\
    $15$ & $8$ & $50$ & $6000$ & $1715$ & $5$ & $35$ \\
    $16$ & $6$ & $26$ & $2496$ & $500$ & $4$ & $20$ \\
    $16$ & $11$ & $46$ & $8096$ & $1620$ & $4$ & $36$ \\
    $25$ & $7$ & $37$ & $6475$ & $1080$ & $5$ & $30$ \\
    $28$ & $10$ & $37$ & $10360$ & $1152$ & $4$ & $32$ \\
    $36$ & $8$ & $50$ & $14400$ & $2058$ & $6$ & $42$ \\
    \hline \\
    \end{tabular}
    \caption{Parameters for flag-transitive, $3$-chain-imprimitive $2$-designs $\Dft(\mathbf{e};k)$ in Construction \ref{constr:gen} with $\mathbf{e}=(e_1, e_2, e_3),$ and $e_1, e_2, e_3 \leq 50$}
    \label{tab:flag-tr}
\end{table}



\begin{thebibliography}{}

\bibitem{grids22}
S.H. Alavi, A. Daneshkhah, A. Devillers and C.E. Praeger,
Block-transitive two-designs based on grids. 
Bull. London Math. Soc.  Published Online November 11, 2022. 
arXiv:2201.01143v1 [math.CO].

\bibitem{chainspaper}
C. Amarra, A. Devillers, and C.E. Praeger,
Block-transitive $2$-designs with a chain of imprimitive partitions, preprint submitted, 2023. 
arxiv:2303.11655 [math.CO]

\bibitem{BPRS}
Bailey, R. A.; Praeger, Cheryl E.; Rowley, C. A.; Speed, T. P.
Generalized wreath products of permutation groups. 
Proc. London Math. Soc. (3) 47 (1983), no. 1, 69--82.


\bibitem{magma}
W. Bosma, J. Cannon, and C. Playoust,
The Magma algebra system. I. The user language.
J. Symbolic Comput., 24 (1997), 235–265.

\bibitem{grids21}
S. Brai\'{c}, J. Mandi\'{c}, A. \v{S}uba\v{s}i\'{c}, T. Vojkovi\'{c} and T. Vu\v{c}i\v{c}i\'{c},
Groups $S_n \times S_m$ in Construction of Flag-transitive Block Designs.
Glas. Mat. 56(76) (2021), 225--240.

\bibitem{CP93}
P.J. Cameron and C.E. Praeger,
Block-transitive $t$-designs I: point-imprimitive designs.
Disc. Math. 118 (1993), 33--43.

\bibitem{CP16}
P. J. Cameron and C. E. Praeger,
Constructing flag-transitive, point-imprimitive designs.
J. Algebraic Combin., 43(4) (2016) 755--769.

\bibitem{CZ}
J. Chen and S. Zhou,
Flag-transitive, point-imprimitive $2$-$(v,k,\lambda)$ symmetric designs with $\lambda$ and $k$ prime powers.
Des. Codes Cryptogr. 89 (2021), no. 6, 1255--1260.

\bibitem{D87}
H. Davies,
Flag-transitivity and primitivity.
Disc. Math. 63 (1987), 91--93.


\bibitem{DLPX}
A. Devillers, H. Liang, C. E. Praeger, and B. Xia, On flag‐transitive $2$-$(v, k, 2)$ designs, \emph{J. Combin. Theory Ser. A} {\bf177} (2021), Article 105309

\bibitem{DP21}
A. Devillers and C.E. Praeger,
On flag-transitive imprimitive 2-designs.
J. Comb. Des. 29(8) (2021), 552--574.

\bibitem{DP22}
A. Devillers and C.E. Praeger,
Analysing flag-transitive point-imprimitive $2$-designs.
Algebraic Combinatorics,  6, 1041--1055.
doi: 10.5802/alco.297    ArXiv:   2208.12455e


\bibitem{LPR}
M. Law, C. E. Praeger, and S. Reichard,
Flag-transitive symmetric $2-(96,20,4)$-designs. 
J. Combin. Theory Ser. A 116 (2009), no. 5, 1009--1022.

\bibitem{MS}
J. Mandi\'{c} and A. \v{S}uba\v{s}i\'{c},
Flag--transitive and point--imprimitive symmetric designs with $\lambda\leq 10$.
J. Combin. Theory, Ser. A, 189 (2022), Article 105620.

\bibitem{M}
A. Montinaro,
Flag-transitive, point-imprimitive symmetric $2-(v,k.\lambda)$ designs with $k > \lambda(\lambda-3)/2$.
ArXiv:  2203.09261.

\bibitem{M2}
A. Montinaro,
The symmetric $2$-$(v,k,\lambda)$ designs, with $k > \lambda(\lambda - 3)/2$, admitting a flag-transitive, point-imprimitive group are known.
ArXiv: 2212.08893.

\bibitem{OR}
E. O'Reilly-Regueiro, On primitivity and reduction for flag-transitive symmetric designs, J.
Combin. Theory Ser. A 109 (2005), 135--148.


\bibitem{P}
C. E. Praeger, The flag-transitive symmetric designs with $45$ points, blocks of size $12$, and $3$ blocks on every point pair. Designs Codes Crypt. 44 (2007), no. 1-3, 115--132.

\bibitem{PS}
C.E. Praeger and C. Schneider,
Permutation Groups and Cartesian Decompositions, London Mathematical Society Lecture Note Series 449 (2018), Cambridge University Press.


\bibitem{PZ}
C. E. Praeger and S. Zhou,
Imprimitive flag-transitive symmetric designs.
J. Combin. Theory Ser. A 113 (2006), 1381--1395.

\bibitem{ZZ}
Y. Zhao and S. Zhou,
Flag-transitive $2$-$(v,k,\lambda)$ designs with $r > \lambda(k-3)$.
Designs, Codes and Crypt. 90 (2022), 863--869.

\end{thebibliography}
\end{document}